\documentclass{amsart}
\usepackage{amssymb, latexsym}

\DeclareMathOperator{\im}{im}
\DeclareMathOperator{\M}{\mathcal{M}}
\DeclareMathOperator{\K}{\mathcal{K}}

\DeclareMathOperator{\rank}{rank}

\def\matrices{M_{m\times n}(K)}

\theoremstyle{plain}
\newtheorem{theorem}{Theorem}
\newtheorem{corollary}{Corollary}
\newtheorem{lemma}{Lemma}

\theoremstyle{definition}

\begin{document}
\title[Constant rank subspaces of matrices]
{A dimension bound for constant rank\\ 
subspaces of matrices over a finite field}

\author[R. Gow]{Rod Gow}
\address{School of Mathematical Sciences\\
University College Dublin\\
 Ireland}
\email{rod.gow@ucd.ie}

\keywords{matrix, rank, constant rank subspace, skew-symmetric matrix, character.}
\subjclass{15A03, 15A33}

\begin{abstract} 
Let $K$ be a field and let $\matrices$  denote the space of $m\times n$ matrices over $K$, where $m\leq n$. We say that  a non-zero subspace $\M$ of $\matrices$ is a constant rank $r$ subspace if each non-zero element of
$\M$ has rank  $r$, where $1\leq r\leq m$. We show in this paper that if $K$ is a finite field with $|K|\geq r+1$, any constant
rank $r$ subspace of $\matrices$ has dimension at most $n$. 
\end{abstract}
\maketitle
\section{Introduction}
\noindent  Let $K$ be a field and, when
$q$ is a power of a prime, let $\mathbb{F}_q$ be the finite field of order $q$. 
Let $m$ and $n$  be positive integers with $m\le n$. Let $\matrices$ denote the vector space of $m\times n$ matrices with entries in $K$. When $m=n$, we write 
$M_n(K)$ in place of $M_{n\times n}(K)$. The \emph{rank} of an element $A$ of $\matrices$ is an integer $s$ satisfying $0\leq s\leq m$, and, of course, $s=0$ if and only if $A$ is the zero matrix. 

Let $\M$ be a non-zero $K$-subspace of $\matrices$ and let $r$ be a positive integer satisfying $r\leq m$. We say that $\M$ is a \emph{constant rank} $r$ subspace if the rank of each non-zero element of $\M$ is $r$. A question that has received considerable attention in the research literature during several decades is the following: what is the largest dimension, as a function of $m$, $n$ and $r$, of a constant rank
$r$ subspace of $\matrices$? As might be expected, the answers depend on the field $K$. 

On the basis of various examples and general constructions, it is tempting to conjecture that
a constant rank $r$ subspace of $\matrices$ has dimension at most $n$. This bound would be optimal for finite fields, since if $r$ is a positive integer not greater than $m$, $M_{m\times n}(\mathbb{F}_q)$ contains
a constant rank $r$ subspace of dimension $n$. See, for example, Theorem 4 of \cite{DGMS}. 

However, it has been known since 1999 that this conjecture is false.
Nonetheless, at the time of writing, we know of only four counterexamples, occurring in $M_3(\mathbb{F}_2)$, $M_4(\mathbb{F}_2)$, 
$M_5(\mathbb{F}_2)$ and
$M_{4\times 5}(\mathbb{F}_2)$ (they are constant rank 2 or 3 subspaces, and the $M_5(\mathbb{F}_2)$ example
is obtained from the $M_{4\times 5}(\mathbb{F}_2)$ example by adding a zero row to each matrix). See, for example, \cite{B}, \cite{Bo} and \cite{DGMS}, 
Example 1. We note that the field involved
in these counterexamples is the smallest one, $\mathbb{F}_2$. This is partly because counterexamples are usually found by exhaustive computer searches of matrices, which are almost impossible to implement for fields of size larger than 3. It is also suspected that any counterexamples
to the conjectured upper bound of $n$ for the dimension of a constant rank $r$ space only occur over fields of size at most $r$.

In 1990, Beasley and Laffey published a paper, \cite{BL}, in which it was claimed that the maximum dimension of a constant rank
$r$ subspace of $\matrices$ is $n$, provided that $|K|\geq r+1$. In a subsequent erratum, published in 1993, the authors noted that their
theorem was only proved subject to the additional assumption that $n\ge 2r-1$. We are not sure of the nature of the error in the argument, as this was not specified in the erratum, and we had difficulty following the proof.

The purpose of the present paper is to prove that the original theorem of Beasley and Laffey is true, provided $K$ is finite, and $|K|\geq r+1$.
The restriction on the size of $K$ seems difficult to omit, and as we noted above, some condition on $|K|$ is required. Our proof
of the upper bound is rather different from that of Beasley and Laffey. We use a general principle that holds for constant rank subspaces
over all fields and then apply a simple counting technique for our main
contradiction.

\section{Proofs of main results}

\noindent 

\noindent We first note that it suffices to prove our dimension bound for a constant rank subspace of $\matrices$ in the case that $m=n$. For let
$\M$ be a constant rank $r$ subspace of $\matrices$, and suppose that $m<n$. We extend each element $A$ of $\M$ to an element $A'$ of $M_n(K)$
by adding an additional $n-m$ zero rows to the bottom of $A$. It is clear that $A$ and $A'$ have the same rank. Let $\M'$ be the subset of
$M_n(K)$ consisting of all such elements $A'$, as $A$ ranges over $\M$. $\M'$ is a $K$-subspace of $M_n(K)$ and $\dim \M'=\dim \M$. Thus, if we can prove that $\dim \M'\leq n$, we will have the same bound for $\dim \M$. 

The following result plays a key role in many investigations concerning the rank of elements in subspaces of matrices. It is proved in
\cite{Fill}, Lemma 1, and \cite{Sweet}, Theorem 1. Both of these papers give a necessary lower bound for $|K|$ larger than the one we quote, 
but the authors of each paper note in separate comments that the result holds for $|K|\geq r+1$, as we may easily verify by studying the proofs.

\begin{lemma} \label{image_of_kernel} Let $K$ be a field and let $\M$ be a non-zero subspace of $M_n(K)$. Let $A\in \M$ have the property that
$r=\rank A\geq \rank B$ for all $B\in \M$. Then if $|K|\geq r+1$, $B u\in \im A$ for each element $u$ of $\ker A$ and each element $B$ of
$\M$. (Note that $\im A$ denotes the image space of $A$.)
 
\end{lemma}

In the counterexamples to the dimension bound for constant rank subspaces, Lemma \ref{image_of_kernel} does not hold. Thus the size of the field is vital in this type of result.

\begin{lemma} \label{dimension_bound_for_kernel}
Let $\M$ be a constant rank $r$ subspace of $M_n(K)$ and suppose, if possible, that $\dim \M=n+1$. 
Let $u$ be a non-zero element of $K^n$ and let $\K_u=\{ A\in \M: Au =0\}$. Then if $|K|\geq r+1$, we have $\dim \K_u\geq n+1-r$.
\end{lemma}

\begin{proof} We define a $K$-linear transformation $T_u:\M\to K^n$ by setting
\[
 T_u(B)=Bu
\]
for all $B\in \M$. It is clear that $\K_u=\ker T_u$. Since $\dim \M>n$, the rank-nullity theorem implies that $\K_u\neq 0$. Thus there exists a non-zero element
$A$, say, of $\M$ with $Au=0$. Since $\M$ is a constant rank $r$ subspace, $|K|\geq r+1$ and $u\in \ker A$, Lemma \ref{image_of_kernel} implies that $\M u\leq \im A$. Now, since $A$ has rank $r$, $\dim \im A=r$ and thus we deduce that
\[
 \dim \M u=\dim T_u(\M)\leq r.
\]
However, we also have
\[
 n+1=\dim \M=\dim T_u(\M)+\dim \K_u\leq r+\dim \K_u.
\]
It follows that $\dim \K_u\geq n+1-r$, as required.
\end{proof}

We can now proceed to the proof of our main theorem.

\begin{theorem} \label{largedimension} Let $q$ be a power of a prime and let $\M$ be a constant rank $r$ subspace of $M_n(\mathbb{F}_q)$, where
$1\leq r\leq n$. Suppose that $q\geq r+1$. Then we have $\dim \M\leq n$.
\end{theorem}

\begin{proof} The result is clear if $r=n$ (and requires no assumption that $q\geq r+1$). Thus, we will assume that $r<n$. Suppose if possible
that $\dim \M>n$. Then we may as well assume that $\dim \M=n+1$, and proceed to derive a contradiction.

Let $\Omega$ be the set of pairs $(A,u)$, where $A$ is a non-zero element of $\M$, and $u$ is a non-zero element of $\mathbb{F}_q^n$ with $Au=0$.
We evaluate $|\Omega|$ by a double counting procedure.  Since each element $A$ of $\M$ has rank $r$, there are $q^{n-r}-1$ non-zero elements
$v$ with $Av=0$. Thus, as $\M$ contains $q^{n+1}-1$ non-zero elements, we see that
\[
 |\Omega|=(q^{n+1}-1)(q^{n-r}-1).
\]
Let $u$ be any non-zero element of $\mathbb{F}_q^n$. The number of non-zero elements $B\in\M$ satisfying $Bu=0$ is $q^{r(u)}-1$, where
$r(u)=\dim \K_u$, and $\K_u$ is defined in Lemma \ref{dimension_bound_for_kernel}. We note that as $q\geq r+1$, $r(u)\geq n+1-r$ by the same lemma.
We thus obtain
\[
 |\Omega|=\sum_{u\neq 0}(q^{r(u)}-1)=-(q^n-1)+\sum_{u\neq 0} q^{r(u)}.
\]

Our two counts yield
\[
 (q^{n+1}-1)(q^{n-r}-1)=-(q^n-1)+\sum_{u\neq 0} q^{r(u)}
\]
and we can rearrange the equation to see that
\[
 q^{2n+1-r}-q^{n-r}-q^{n+1}+q^n=\sum_{u\neq 0} q^{r(u)}.
\]
Clearly, $q^{n-r}$ divides the left hand side, but $q^{n+1-r}$ does not divide it. On the other hand, since $r(u)\geq n+1-r$ for all non-zero
$u$, $q^{n+1-r}$ divides the right hand side. This is a contradiction. Thus, $\dim \M\leq n$, as required.

\end{proof}

\begin{corollary} \label{general_theorem}
 
Let $q$ be  a power of a prime and let $\M$ be a constant rank $r$ subspace of $M_{m\times n}(\mathbb{F}_q)$, where
$1\leq r\leq m$. Suppose that $q\geq r+1$. Then we have $\dim \M\leq n$.
\end{corollary}

Note that Corollary \ref{general_theorem} holds without restriction on $q$ when $r=m$, as may easily be proved without use of counting arguments.
For arbitrary $q$, we showed in \cite{DGMS}, Theorem 6, that a constant rank $r$ subspace of $M_{m\times n}(\mathbb{F}_q)$ has dimension at most
$m+n-r$. This result can be proved by elementary double counting arguments, without recourse to the use of characters of finite abelian
groups. This general upper bound is occasionally optimal in cases when the hypothesis of Corollary \ref{general_theorem} does not hold.

We hope to investigate constant rank subspaces of symmetric matrices over $\mathbb{F}_q$ using ideas similar to those developed in this paper.
More detailed results can be obtained in the symmetric case, as we have more arithmetic information coming from the geometry of symmetric bilinear forms.


\begin{thebibliography}{9}

\bibitem{B} L. B. Beasley, \emph{Spaces of rank-$2$ matrices over ${\rm GF}(2)$}, Electron. J. Linear Algebra {\bf 5} (1999), 11-18.
\smallskip
\bibitem{BL} L. B. Beasley and T. J. Laffey, \emph{Linear operators on matrices: the invariance of
rank-$k$ matrices},  Linear Algebra Appl. {\bf 133} (1990), 175-184. Erratum, Linear Algebra Appl. {\bf 180} (1993), 2.
\smallskip
\bibitem{Bo} N. Boston, \emph{Spaces of constant rank matrices over $GF(2)$}, Electron. J. Linear Algebra {\bf 20} (2010), 1-5.
\smallskip
\bibitem{DGMS} J.-G. Dumas, R. Gow, G. McGuire and J. Sheekey, \emph{Subspaces of matrices with special rank properties}, Linear Algebra Appl. {\bf 433} (2010), 191-202.
\smallskip
\bibitem{Fill} P. Fillmore, C. Laurie and H. Radjavi, \emph{On matrix spaces with zero determinant}, Linear and Multilinear
Algebra {\bf 18} (1985), 255-266.
\smallskip
\bibitem{Sweet} L. G. Sweet and J. A. MacDougall, \emph{The maximum dimension of a subspace of nilpotent matrices of index $2$}, Linear Algebra. Appl. {\bf 431} (2009), 1116-1124.


\end{thebibliography}
\end{document}